\documentclass[reqno,12pt]{amsart}
\author{Julia Brandes}
\title{Sums and differences of power-free numbers}
\address{Mathematisches Institut, Bunsenstr. 3--5, 37073 G{\"o}ttingen, Germany}
\email{jbrande@uni-math.gwdg.de}

\usepackage[english,ngerman]{babel}
\addto\extrasenglish{\languageshorthands{ngerman}}
\usepackage[T1]{fontenc}
\usepackage[latin1]{inputenc}
\usepackage{latexsym}
\usepackage{epsfig}
\usepackage{bbm}
\usepackage{amssymb}
\usepackage{amsmath}
\usepackage{verbatim}
\usepackage{amsthm}
\usepackage{enumerate}
\usepackage[hmargin=3.5cm,vmargin=3.5cm]{geometry}

\usepackage{bm}

\def\res#1#2#3{ {\left( {\frac{#2}{#3}} \right) \negthickspace }_{#1} }
\def\dsum#1#2{ \sum_{\substack{ {#1} \\ {#2} }} }

\def\N{\mathbb{N}}
\def\mmod#1{\;(\mathrm{mod}\;{#1})}
\def\eps{\varepsilon}
\DeclareMathOperator{\card}{Card}

\newtheorem{thm}{Theorem}
\newtheorem{lem}{Lemma}

\theoremstyle{definition}

\begin{document}
\selectlanguage{english}

\begin{abstract}
We employ a generalised version of Heath-Brown's square sieve in order to establish an asymptotic estimate of the number of solutions $a, b$ to the equations $a+b=n$ and $a-b=n$, where $a$ is $k$-free and $b$ is $l$-free. This is the first time that this problem has been studied with distinct powers $k$ and $l$. 
\end{abstract}
\maketitle

\section{Introduction}

Let $A_k \subseteq \mathbb N$ denote the set of $k$-free numbers. We are interested in the number $R^{-}_{k,l}(x;n)$ of twins $a, a-n \leq x$ with a given distance $n$, where $a$ is $k$-free and $a-n$ is $l$-free. A closely related question asks for the number $R^{+}_{k,l}(x;n)$ of solutions $a \in A_k$ and $b \in A_l$ to the equation $a+b=n$, where we may impose the additional size restriction that one of the summands be bounded by a parameter $x$. Observe that trivially $R^-_{k,l}(x;n)=0$ when $n\geq x$ and $R^+_{k,l}(x;n) = R^+_{k,l}(n;n)$ when $x\geq n$. More precisely, $a$ is restricted to an interval of length $X^{\pm}(x;n)$, where $X^{-}(x;n) = \max\{0,x-n\}$ and $X^{+}(x;n) = \min\{x,n\}$, so as the $k$-free numbers are of positive density in the integers,  we expect the main term to be proportional to $X^{\pm}(x;n)$.

When the exponents $k$ and $l$ are equal these are well-studied problems, and it is not hard to establish an asymptotic formula. In fact, if $\omega(k)$ denotes the least number $\omega$ satisfying 
\begin{align*}
  R^{\pm}_{k,k}(x;n) = X^{\pm}(x;n) \prod_{\substack{p \text{ prime}\\ p^{k}|n}}\left( \frac{p^{k}-1}{p^{k}-2} \right) \prod_{p \text{ prime}}(1-2p^{-k})  + O\left( (x+n)^{\omega+ \eps}\right),
\end{align*}
then elementary methods (see Carlitz \cite{carl} for $R^{-}_{2,2}(x;1)$, and  Evelyn and Linfoot \cite{evelyn2} and Estermann \cite{estermann} for $R^{+}_{k,k}(n;n)$) suffice to show ${\omega(k) \leq 2/(k+1)}$. Interestingly enough, this still seems to be the state of the art in the case of $R^{+}_{k,k}(n;n)$, while for $R^{-}_{k,k}(x;1)$ better error terms are available. In fact, Heath-Brown \cite{hb} used sieve methods to establish $\omega(2) \leq 7/11$, an approach that has later been generalised to higher powers in the author's diploma thesis \cite{diplom}, yielding an error term\footnote{The original work contains an oversight in the numerical analysis, which led to a different exponent. This error is corrected in the present paper.} of
\begin{align*}
\omega(k) \leq \begin{cases} 14/(7k+8) & \text{ if } k \leq 4, \\[2mm] \displaystyle{\frac{8k+3}{4k^2+6k+2}} & \text{ if } k \geq 5.\end{cases}
\end{align*}
The advent of the determinant method triggered further improvements. Dietmann and Marmon \cite{rd+marmon} established $\omega(k) \leq 14/(9k) $, while independently Reuss~\cite{reuss} reduced the exponent to $\omega(2) \leq (26+ \sqrt{433})/81 \approx 0.578$ and ${\omega(k) \leq 169/(144k)}$ for larger values of $k$. Both results are much stronger than~\cite{diplom}.\\

It is now natural to ask whether these methods continue to be applicable when the powers are distinct. However, except for a passing remark by Br\"udern et al. \cite{jb++} claiming that their approach to study $R^+_{k,k}(n;n)$ via the circle method can be adapted to distinct powers, the author is not aware of any previous work regarding this more general setting. 
In this note, we show how the power sieve as developed by Heath-Brown \cite{hb} in the case $k=2$ and extended to higher powers independently by the author \cite{diplom} and Munshi \cite{munshi} can be adapted to obtain an asymptotic formula for $ R^{\pm}_{k,l}(x;n)$. 
Let
\begin{align*}
  C_{k,l}(n)=\prod_{\substack{p \text{ prime}\\ p^{k}|n}}\left( \frac{p^{k+l}-p^{l}}{p^{k+l}-p^k-p^l} \right) \quad \text{ and } \quad \mathfrak S_{k,l}=\prod_{p \text{ prime}}(1-p^{-k}-p^{-l}),
\end{align*}
then the result is the following.
\begin{thm}\label{thm}
Suppose $2 \leq k\leq l$, then the counting functions $R^{+}_{k,l}(x;n)$ and $R^{-}_{k,l}(x;n)$ obey the asymptotic formula
    \begin{align*}
      R^{\pm}_{k,l}(x;n) = X^{\pm}(x;n) C_{k,l}(n) \mathfrak S_{k,l} + O\big((x+n)^{\omega(k,l) + \eps}\big),
    \end{align*}
    where the exponent $\omega(k,l)$ in the error term is bounded by
    \begin{align*}
      \omega(k,l) \leq \begin{cases}
         \displaystyle{\frac{4(k+l)+3}{2(2k+1)(l+1)}}& \text{ if }  k \leq l \leq (4/3)(k-1),  \\[3mm]
         \displaystyle{\frac{7(k+l)}{6k+2l+7kl}} & \text{ if } (4/3)(k-1) \leq l \leq 4k,  \\[3mm]
         \displaystyle{\frac{k+l}{k(l+2)}} & \text{ if } l \geq 4k,
      \end{cases}
      \end{align*}
      and the implied constant depends on $k$, $l$ and $\eps$, but is independent of $x$ and $n$.
In the case $k>l$ the same asymptotic formula is true with the roles of $k$ and $l$ reversed.
\end{thm}
Observe that Theorem~\ref{thm} reproduces the result from \cite{diplom} and extends it to the problem of estimating $R^{+}_{k,k}(x;n)$, in which case it is the first improvement on the error term since Estermann's work from 1931 \cite{estermann}. For $x \ll n$ the error remains non-trivial as long as $n^{\omega(k,l)} \ll x$.\\

A word is in order as to why we believe this result to be competitive, in particular in the light of the strong bounds that can be obtained by the determinant method in the case $k=l$. Observe that for $l \geq k$ every $k$-free number is \emph{a forteriori} also $l$-free, so when $l$ is not much larger than $k$ the bounds on $\omega(k)$ obtained by Reuss \cite{reuss} will yield a stronger error term than that of Theorem~\ref{thm}. However, in its current formulation the determinant method fails to properly accommodate distinct powers as it seems to rely crucially on rewriting the equation $u^kh - v^kj=n$ in the shape $(u/v)^{k} - j/h = n/(v^kh)$ and then counting rational points within a distance $O_n((v^{k}h)^{-1})$ from the curve $\alpha^k = \beta$, a transformation impossible when $k$ and $l$ are distinct. This means that when $l$ is sufficiently large in terms of $k$ we should expect Theorem~\ref{thm} to yield a stronger error term than what is being obtained by the determinant method, and indeed, our exponent supersedes that of Reuss~\cite{reuss} as soon as $l$ is larger than roughly $144k/25$. 
Furthermore, in the limit $l \rightarrow \infty$ the exponent in Theorem~\ref{thm} tends to $1/k$, thus replicating the known error term in the asymptotic formula for the number of $k$-free integers. For comparison, the methods of Br\"udern et al. \cite{jb++} yield an error term whose exponent is bounded below by $1/2$.\\

Throughout the argument the following conventions will be obeyed. No effort will be made to track the exact `value' of infinitesimal quantities arising in the analysis, so all statements involving $\eps$ are true for any $\eps >0$. The implied constants will be allowed to depend on all parameters except $n$ and $x$. We write $d(n)$ for the number of divisors of an integer $n$, and $(a,b)$ denotes the greatest common divisor of two integers $a$ and $b$. Furthermore, since the analysis differs only very little for $R^+$ and $R^-$ we will write $R^{\pm}$ and use the ambiguous signs $\pm$ and $\mp$ throughout. This will happen in a consistent way, so in $\mp$ the negative sign corresponds to $R^+$ and the positive value arises when $R^-$ is considered.  \\

The author is grateful to the referee for valuable comments.

\section{Elementary considerations}

Recall that our counting function is given by 
\begin{align*}
  R^{\pm}_{k,l}(x;n) &= \left\{a \leq x: \, a \in A_k, \pm(n-a) \in A_l \right\}\hspace{-.5mm}.
\end{align*}
Using the indicator function on $k$-free numbers $\mu_k(n)= \sum_{d^k|n} \mu(d)$, we may rewrite this as
\begin{align*}
		R^{\pm}_{k,l}(x;n) &= \sum_{u,v \leq x} \mu(u) \mu(v) N^{\pm}_{k,l}(x, n; u,v),
\end{align*}
where 
\begin{align*}
  N^{\pm}_{k,l}(x, n; u,v) = \card \{ 1 \leq a \leq x :\;  \pm(n-a)\geq 1, \, u^k|a,\, v^l|(n - a) \}.
\end{align*}
One sees that with our assumption $k \le l$ one has $N^{\pm}_{k,l}(x,n; u,v)=0$ when $(u,v)^{k}$ does not divide $n$ and 
\begin{align*}
N^{\pm}_{k,l}(x, n; u,v) &= X^{\pm}(x; n)\frac{(u,v)^{k}}{u^k v^l} + O(1)
\end{align*}
otherwise. 

For a parameter $y$ to be chosen optimally later the contribution with $uv\leq y$ is given by
\begin{align}\label{uv<y}
   X^{\pm}(x; n)  \hspace{-1mm}\dsum{u,v \in \mathbb N}{(u,v)^{k}|n} \frac{\mu(u) \mu(v)(u,v)^k}{u^k v^l} + O\Bigg(X^{\pm}(x; n) \hspace{-1mm} \dsum{uv>y}{(u,v)^k|n} \frac{(u,v)^k}{u^k v^l}\Bigg) + O\Big(\sum_{uv \leq y}1\Big).
\end{align}
We may rewrite the sum in the main term by extracting common factors and writing $u=u'd$ and $v=v'd$. Note that $\mu(u'd)$ takes a non-zero value only if $u'$ and $d$ are coprime, so we have
\begin{align*}
		  \sum_{(u,v)^{k}|n} \frac{\mu(u) \mu(v)(u,v)^k}{u^k v^l} &= \sum_{d^{k}|n} \frac{\mu^2(d)}{d^{l}} \dsum{(u', v')=1}{(d,u')=(d,v')=1} \frac{\mu(u') \mu(v')}{{u'}^k {v'}^l}\\
		&= \sum_{d^{k}|n} \frac{\mu^2(d)}{d^{l}} \sum_{(h,d)=1} \frac{\mu(h)}{h^k} \sum_{w|h} \frac{1}{w^{l-k}}.
\end{align*}
The sum over $h$ is multiplicative with an Euler product expansion given by
\begin{align*}
  \prod_{p\nmid d}(1-p^{-k}-p^{-l}) = \mathfrak S_{k,l} \prod_{p|d}(1-p^{-k}-p^{-l})^{-1},
\end{align*}
and after expanding the $d$-summation one recovers the constant $C_{k,l}(n)$
for $k \leq l$. The respective result for $l>k$ follows by symmetry. \\

We now turn to the error
\begin{align*}
		 E^{\pm}_{k,l}(x;n) = R^{\pm}_{k,l}(x;n) - X^{\pm}(x; n) \mathfrak S_{k,l} C_{k,l}(n).
\end{align*}
 Using that $X^{\pm}(x; n)\leq x$, we have
\begin{align*}
		 E^{\pm}_{k,l}(x;n) \ll  \sum_{uv \leq y} 1 + x\dsum{uv>y}{(u,v)^k|n} \frac{ (u,v)^k} {u^k v^l} + \dsum{u,v \leq x}{uv>y} \mu(u) \mu(v)N^{\pm}_{k,l}(x,n; u,v),
\end{align*}
where the first two terms are the error terms from \eqref{uv<y} and the last term accounts for the contribution not considered in \eqref{uv<y}. Observe that the first term is $O(y \log y)$ trivially and the second term is bounded above by
\begin{align*}
		  x\dsum{uv > y}{(u,v)^{k}|n} \frac{(u,v)^k}{{u^k v^l}} &\ll x\sum_{d^k|n}d^{-l} \sum_{v=1}^{\infty}v^{-l} \dsum{u \in \N}{u>y/(d^2v)}u^{-k}.
\end{align*}
The inner sum in the above expression is $\ll (d^2v/y)^{k-1}$ if $v\leq y/d^2$ and $\ll 1$ else, whence on the right hand side we have
\begin{align*}
  x\sum_{d^k|n}d^{-l} \sum_{v=1}^{\infty}v^{-l} \dsum{u \in \N}{u>y/(d^2v)}  u^{-k} &\ll xy^{1-k}\sum_{d^k|n}d^{-l+2k-2} \sum_{1 \leq v \leq y/d^2}v^{k-l-1} \\ 
  & \qquad+ x\sum_{d^k|n}d^{-l} \dsum{v \in \N}{v>y/d^2}v^{-l}.
\end{align*}
The $v$-summation in the first term on the right hand side is empty unless $d\leq \sqrt y$, and in the second term we have another case distinction as before according to whether $y/d^2>1$ or not.
It follows that the above is 
\begin{align*}
  &\ll xy^{1-k+\eps}\dsum{d\leq \sqrt y}{d^k|n}d^{-l+2k-2} + xy^{1-l}\dsum{d\leq \sqrt y}{d^k|n}d^{l-2} + x\dsum{d>\sqrt y}{d^k|n}d^{-l}\\
  &\ll xy^{1-k+\eps}\big(1+y^{(2k-l-2)/2}n^{\eps}\big)+ xy^{-l/2}n^{\eps}   \\
  &\ll xy^{1-k+\eps} + xy^{-l/2}n^{\eps}.
\end{align*}
This yields the overall error
\begin{align*}
		 E^{\pm}_{k,l}(x;n)\ll y^{1+\eps} + xy^{1-k+\eps} + xy^{-l/2}n^{\eps} +   \dsum{uv>y}{u,v \le x} \mu(u) \mu(v) N^{\pm}_{k,l}(x, n; u,v).
\end{align*}
The last term will be the focus of our analysis in the following section.

\section{Sieving for powers}
We split the terms with $uv>y$ in roughly $(\log x)^2$ intervals $U < u \leq 2U$ and $V< v \leq 2V$. Then there exist values $U$ and $V$ such that
\begin{align*}
		\sum_{uv>y} \mu(u) \mu(v) N^{\pm}_{k,l}(x,n; u,v) \ll N (\log x)^2,
\end{align*}
where $N$ is given by
\begin{align}\label{eq:def-N}
		N=\card \{ u,v,h,j \!: \, U< u \leq 2U,\, V< v \leq 2V, \, hu^k \pm jv^l =n \}.
\end{align}
Let $N_h$ denote the number of $u,v,j$ counted by $N$ for a fixed value of $h$, and write
\begin{align*}
		N(H) = \sum_{H < h \leq 2H}N_h.
\end{align*}
There are $O(\log x)$ intervals $H < h \leq 2H$ with
\begin{align}\label{eq:U^kH}
		H \ll xU^{-k},
\end{align}
so for some $H$ one has $N \ll  N(H) \log x$ and consequently
\begin{align}\label{eq:error1}
		 E^{\pm}_{k,l}(x;n)\ll y^{1+\eps} + xy^{1-k+\eps} + xy^{-l/2}n^{\eps} + (\log x)^3 N(H).
\end{align}

We may assume without loss of generality that
\begin{align}\label{eq:U^k>V^l}
		U^k \geq V^l.
\end{align}
This might involve swapping the parameters $h$, $u$ and $k$ with $j$, $v$ and $l$, respectively, which in the situation of counting $a-b=n$ inserts a minus sign, so we let $\sigma=-1$ if this sign change occurs and $\sigma=+1$ else. Then, having fixed a suitable value $H$, let
\begin{align*}
		J_1 &= \min \{1,  \pm 2^{-l}V^{-l}(\sigma n-2^{k+1}HU^k)\}, & J_2 &=  \pm V^{-l}(\sigma n-HU^k).
\end{align*}
Note that by construction we have $J_1 < J_2$.
It follows that
\begin{align*}
		N_h& \leq \card \{J_1 \leq j \leq J_2,\, V< v \leq 2V, \, U< u \leq 2U \!: \, hu^k \pm jv^l = \sigma n \}.
\end{align*}

Let $w_k(z)=0$ for all $z$ that are not divisible by $h^{k-1}$, and
\begin{align*}
		w_k(h^{k-1}a) = \card \{ (v,j) \! : \, h|(\sigma n \mp v^lj), a=\sigma n \mp v^lj \},
\end{align*}
where $V < v \leq 2V$ and $ J_1 \leq j \leq J_2$. This is legitimate as $\sigma n \mp v^lj >0$ for all admissible values, and we have
\begin{align*}
		N_h \leq \sum_{m=1}^{\infty} w_k(m^k).
\end{align*}
We now apply the power sieve as in \cite{diplom}. Let $Q$ be a parameter satisfying ${\log x \leq Q \leq x}$ which will be chosen optimally later, and let $\mathcal P$ be the set of primes in the interval $[Q,2Q]$ that do not divide $hn$ and are congruent to $1$ modulo $k$, then by the Siegel--Walfisz Theorem~we have $\card \mathcal P \asymp Q/\log Q$. Write further $\big(\frac{n}{p}\big)_{\hspace{-.5mm}k}$ for a non-principal character modulo $p$ whose $k$-th power is the principal character modulo $p$. That such a character exists is ensured by our choice of the set $\mathcal P$.  The power sieve as in \cite[Lemma~2.1]{munshi} or {\cite[Theorem~1]{diplom}} now yields
\begin{align}\label{eq:sifting}
		N_h & \ll \frac{\log Q}{Q} \sum_{z=1}^{\infty} w_k(z) + \max_{\substack{p, q \in \mathcal P \\ p \neq q}}\left| \sum_{v,j} \res{k}{h^{k-1}(\sigma  n \mp v^lj)}{p} \overline{\res{k}{h^{k-1}(\sigma  n \mp v^lj)}{q}} \right|\hspace{-.5mm},
\end{align}
where the sum is over all $V< v \leq 2V$, $J_1 \leq j \leq J_2$ satisfying $ h| (\sigma  n \mp v^lj)$.
The first term can be estimated by
\begin{align*}
		\frac{\log Q}{Q} \sum_{z=1}^{\infty} w_k(z) &\ll \frac{\log Q}{Q} \sum_{V< v \leq 2V} \dsum{\xi \leq x}{\xi \equiv \sigma  n \mmod{v^l}} d(\xi) \ll Q^{-1+\eps} V^{1-l} x^{1+\eps},
\end{align*}
so the deeper challenge lies in establishing a good bound for the second term. This analysis will follow the treatment of \cite{hb} and \cite{diplom} closely.

Denote the expression inside the absolute values in the second term of \eqref{eq:sifting} by $S$. We have
\begin{align}\label{eq:S}
		S =  (hpq)^{-2} \sum_{\gamma, \delta =1}^{hpq} S(h, pq; \gamma, \delta) \theta_{\gamma} \phi_{\delta},
\end{align}
where
\begin{align}
		S(h, pq; \gamma, \delta) &= \dsum{\alpha, \beta =1}{h| (\sigma  n \mp \alpha^j \beta)}^{hpq} \res{k}{\sigma  n\mp \alpha^l\beta}{p} \overline{\res{k}{\sigma  n\mp \alpha^l\beta}{q}} e\left(\frac{\gamma\alpha+\delta\beta}{hpq}\right)\hspace{-1mm}, \nonumber \\
		\theta_{\gamma} & = \sum_{V< v \leq 2V} e \left( \frac{-\gamma v}{hpq}\right) \ll \min \left(V, \left\| \frac{\gamma}{hpq} \right\|^{-1} \right)\hspace{-1mm}, \label{eq:theta} \\
		\phi_{\delta} & = \sum_{J_1 \leq j \leq J_2} e \left( \frac{-\delta j}{hpq}\right) \ll \min \left(V^{-l}(U^kH+n), \left\| \frac{\delta}{hpq} \right\|^{-1} \right)\hspace{-1mm}. \label{eq:phi}
\end{align}
Writing further
\begin{align*}
		S_1(p; c,d)&=\sum_{\alpha, \beta=1}^p \res{k}{\sigma  n \mp \alpha^l\beta}{p} e \left(\frac{c \alpha + d \beta}{p} \right)\hspace{-1mm},\\
		S_2(r^i; c,d) &= \dsum{\alpha, \beta=1}{r^i|(\sigma  n \mp \alpha^l\beta)}^{r^i} e \left(\frac{c \alpha + d \beta}{r^i} \right)\hspace{-1mm},
\end{align*}
one has 
\begin{align*}
		S(h, pq; \gamma, \delta) = S_1(p; c,d) \overline{S_1(q; -c,-d)} \prod_{r^i \| h} S_2(r^i; c,d)
\end{align*}
where $h=\prod r^i$ denotes the prime power decomposition of $h$, and $c$ and $d$ are certain integers satisfying $(c, hpq)=(\gamma, hpq)$ and $(d, hpq)=(\delta, hpq)$.
\begin{lem}\label{lem:S}
		Suppose $p$ is a prime number not dividing $n$, then the exponential sums $S_1(p; c,d)$ and $S_2(p; c,d)$ can be estimated by
		\begin{align*}
				S_1(p; c,d) & \ll p, \\
				S_2(p; c,d) & \ll p^{1/2}(p,c,d)^{1/2}.
		\end{align*}
\end{lem}
\begin{proof}
		This is analogous to the treatment in \cite[\S5]{hb} or \cite[Lemma~2]{diplom} and is elementary except in the generic case when $p \nmid cd$, in which case it depends on the Weil conjectures. 
\end{proof}
Let $w$ denote the product of all primes coprime to $n$ that divide $h$ exactly once. By following through the arguments of \cite[p.256sq.]{hb} or \cite[p.12]{diplom} one sees that Lemma~\ref{lem:S} implies
\begin{align*}
		\prod_{r^i \| h} S_2(r^i; c,d) \ll Hw^{-1/2}(w, \gamma, \delta)^{1/2},
\end{align*}
whence we obtain
\begin{align*}
		S(h, pq;\gamma, \delta) \ll Q^2Hw^{-1/2+\eps}(w, \gamma, \delta)^{1/2}.
\end{align*}
Inserting the definitions \eqref{eq:S}, \eqref{eq:theta} and \eqref{eq:phi} of $S$, $\theta_\gamma$ and $\phi_\delta$, respectively, together with  \eqref{eq:U^kH}, one concludes that
\begin{align*}
		S &\ll Q^{-2}H^{-1}w^{-1/2 + \eps} \Bigg( V^{1-l}(x+n) w^{1/2} +VQ^2H\sum_{1 \leq \delta\leq hpq/2} \delta^{-1}(w,\delta)^{1/2}      \\
& \quad + V^{-l}Q^2H(x+n)\sum_{1 \leq \gamma \leq hpq/2}\gamma^{-1} (w, \gamma)^{1/2}   + Q^4H^2 \sum_{\gamma,\delta} (\gamma\delta)^{-1}(w, \gamma, \delta)^{1/2}\Bigg) \\
& \ll  Q^{-2}H^{-1}V^{1-l}(x+n)w^{\eps}  + \left(V^{-l}(x+n)+V+HQ^2 \right)w^{-1/2+\eps} (\log HQ^2)^2.
\end{align*}

We can now estimate our sifted sum $N(H)$. By an argument similar to that in  \cite[p.257]{hb} or \cite[p.13]{diplom} one has $\sum_{H < h \leq 2H} w^{-1/2+\eps} \ll H^{1/2+\eps}$, so $N(H)$ is bounded by
\begin{align*}
	 x^{\eps}\big(Q^{-1} V^{1-l}x  + (x+n)Q^{-2}V^{1-l}+  (V^{-l}(x+n) + V + HQ^2)H^{1/2} \big).
\end{align*}
After recalling \eqref{eq:U^kH}, it follows that
\begin{align}\label{error1}
		N &\ll  Q^{-1} V^{1-l}(x+n)^{1+\eps}  + x^{1/2+\eps}(x+n)U^{-k/2}V^{-l} \nonumber\\ 
		& \qquad+ x^{1/2+\eps} U^{-k/2}V + x^{3/2+\eps}U^{-3k/2}Q^2.
\end{align}
It remains to optimise the parameters in order to establish the theorem.

\section{A game of optimisation}

In this section it is our task to optimise the parameters in the error 
\begin{align*}
 E^{\pm}_{k,l}(x;n)&\ll y^{1+\eps} + xy^{1-k+\eps} + xy^{-l/2}n^{\eps} + x^\eps N,
\end{align*}
where $N$ is bounded by the expression in \eqref{error1}. We will write $z=x+n$ and use $x \ll z$ throughout. 
With this convention the first and the last term of \eqref{error1} are roughly equal if  
\begin{align*}
  Q = (\log x)^2 +  z^{-1/6+\eps}U^{k/2}V^{(1-l)/3}.
\end{align*}
Furthermore, it follows from the assumption made in \eqref{eq:U^k>V^l} that 
\begin{align*}
		(U^k)^{-1/2}V \leq (U^{\lambda k} V^{(1-\lambda)l})^{-1/2}V
\end{align*}
for every $\lambda \in [0,1]$. With $\lambda = (l-2)/(k+l)$ the powers in $U$ and $V$ coincide, and since we had $UV \gg y$ this expression is bounded by $y^{-k\lambda/2}$, yielding 
\begin{align}\label{eq:error-init}
	N&\ll  z^{7/6+\eps}U^{-k/2} V^{2(1-l)/3}  + z^{3/2+\eps}U^{-k/2}V^{-l} + z^{1/2+\eps}y^{\frac{-k(l-2)}{2(k+l)}}.
\end{align}

In order to evaluate the terms depending on $U$ and $V$ one has to make recourse to an auxiliary bound for the quantity given in \eqref{eq:def-N}.
\begin{lem}\label{lem:alt}
		We have
		\begin{align*}
				N \ll x^{1+\eps}(U^{-k}V + U^{1-k}V^{1-l}).
		\end{align*}
\end{lem}
\begin{proof}
		Note that by the definition \eqref{eq:def-N} the quantity $N$ is equal to
		\begin{align}\label{eq:alt}
				N = \sum_{V< v \leq 2V} \sum_{h \leq xU^{-k}} \dsum{U< u \leq 2U}{u^kh \equiv \sigma n \mmod{v^l}} 1.
		\end{align}
    The innermost sum can be estimated via Hensel's Lemma. If $d= (n,v^l)$, the quantity in question is the number of $u \in (U,2U]$ satisfying $d|u^k$ for which $u^kh/d \equiv \sigma n/d \pmod{v^l/d}$. This congruence has at most $k$ solutions for each prime divisor of $v^l/d$, giving altogether an upper bound of $d_k(v)$. On the other hand, one has the trivial bound $UV^{-l}$. 
Inserting this into \eqref{eq:alt} yields the required bound.
\end{proof}

Notice that one has
\begin{align*}
  x^{1+\eps}U^{1-k}V^{1-l} \ll x^{1+\eps} (UV)^{1-\min\{k,l\}} \ll x^{1+\eps}y^{1-\min\{k,l\}},
\end{align*}
which is absorbed in the second term of \eqref{eq:error1}.
We obtain our bound by comparing the remaining term of \eqref{eq:alt} with the first two terms in \eqref{eq:error-init}, respectively. Observe that
\begin{align*}
 \min \big\{ z^{3/2} U^{-k/2} V^{-l}, z U^{-k} V \big\} &  \leq  z^{(3/2)\alpha + (1-\alpha) } U^{-(k/2)\alpha -(1-\alpha)k} V^{-l\alpha + (1-\alpha)} \\
& = z^{1+\alpha/2} U^{k\alpha/2 - k} V^{-(l+1)\alpha+1}.
\end{align*}
As before, it follows from \eqref{eq:U^k>V^l} that for every $\mu \in [0,1]$ one has
\begin{align*}
  U^{k\alpha/2 - k} V^{-(l+1)\alpha+1} \ll U^{\mu k(\alpha/2 - 1)} V^{(1-\mu)l(\alpha/2-1)-(l+1)\alpha+1},
\end{align*}
and on choosing $\mu=\frac{(2+\alpha)l +2(\alpha-1)}{(2-\alpha)(k+l)}$ the powers in $U$ and $V$ coincide and one obtains
\begin{align*}
& \min \big\{ z^{3/2} U^{-k/2} V^{-l}, z U^{-k} V \big\} \leq  z^{1+\alpha/2} y^{k(\alpha/2-1)\mu}.
\end{align*}
Similarly, we have
\begin{align*}
 & \min \big\{ z^{7/6 } U^{-k/2} V^{2(1-l)/3}, z U^{-k} V \big\} \\
& \quad  \leq  z^{1+\beta/6} U^{k\beta/2 - k} V^{-(2l+1)\beta/3+1}\\
& \quad \leq z^{1+\beta/6} U^{\nu k(\beta/2 - 1)} V^{(1-\nu)l(\beta/2-1)-(2l+1)\beta/3+1}
\end{align*}
and the powers of $U$ and $V$ coincide on choosing $\nu=\frac{(6+\beta)l + 2(\beta-3)}{(6-3\beta)(k+l)}$. The bound for $N$ is therefore given by
\begin{align}\label{eq:error-ab}
  N & \ll  z^{1/2+\eps}y^{\frac{-k(l-2)}{2(k+l)}} + z^{1+\alpha/2+\eps} y^{\mu k(\alpha/2-1)} +  z^{1+\beta/6+\eps} y^{\nu k(\beta/2 - 1)}.
\end{align}
Put $y=z^{\omega}$ for some suitable power $\omega<1$. This is legitimate whenever $n^\omega \leq x$, and in the opposite case the statement of the theorem is trivial. Then by inserting the values for $\mu$ and $\nu$ we see that the last two terms of \eqref{eq:error-ab} coincide if $\beta = 3 \alpha$, except possibly in the case when ${\omega=\omega_0= (k+l)/(kl+2k)}$. However, one checks that in this case the two terms coincide for all choices of $\alpha$ and $\beta$.
Also, the last term in \eqref{eq:error-ab} is roughly of size $y$ if
\begin{align*}
  \beta=-\frac{6\omega l (1+k)-6(k+l)}{\omega k(l+2)-(k+l)}.
\end{align*}

It remains to find the minimal $\omega$ for which the parameters $\alpha, \beta, \mu, \nu$ all lie in the unit interval $[0,1]$. Notice that our choice of $\beta$ implies
\begin{align*}
  \mu&=-\frac{(2+l)\omega - 3}{(2k+l+2kl)\omega - 2(k+l)},\\
  \nu&=-\frac{(2+l)\omega - 3}{(2k+3l+4kl)\omega - 4(k+l)}.
\end{align*}

The horizontal asymptotes of $\beta$, $\mu$ and $\nu$ as functions of $\omega$ are negative, so the pre-images of $[0,1]$ under each of these functions are also closed intervals. A modicum of computation shows that the endpoints of these intervals are given by
\begin{align*}
  \omega_{\beta=0}&=\frac{k+l}{kl+l}, &  \omega_{\beta=1}  &=\frac{7(k+l)}{2k+6l+7kl},\\
  \omega_{\mu=0} &=\frac{3}{l+2}, & \omega_{\mu=1}&=\frac{2(k+l)+3}{2(k+1)(l+1)},\\
\omega_{\nu=0} &=\frac{3}{l+2}, & \omega_{\nu=1}&=\frac{4(k+l)+3}{2(k+1)(2l+1)}.
\end{align*}

In order to analyse the intervals defined by these values, one has to consider various cases. The expressions for $\beta$, $\mu$ and $\nu$ are increasing functions in $\omega$ when $l>2k$ and decreasing else. Consider first the case when $l>2k$. Then one has $\omega_{\mu=1}> \omega_{\nu=1}$, so the optimal value for $\omega$ is at the lower endpoint of the interval
\begin{align*}
[\omega_{\nu=0} , \omega_{\nu=1} ]\cap [\omega_{\beta=0} ,\omega_{\beta=1} ].
\end{align*}
One checks that this is non-empty and has its lower endpoint at ${\omega_{\beta=0}=\frac{k+l}{kl+l}}$.
If, on the other hand, $l$ is smaller than $2k$, then one has $\omega_{\mu=1}< \omega_{\nu=1}$
and we want to find the lower endpoint of the interval
\begin{align*}
  [ \omega_{\nu=1},\omega_{\nu=0}  ]\cap [\omega_{\beta=1} , \omega_{\beta=0} ].
\end{align*}
This, too, is always non-empty, and one has $\omega_{\nu=1}< \omega_{\beta=1}$ if $l<(3/4)k+1$ and the opposite inequality whenever $(3/4)k+1<l<2k$.
Finally, in the case $l=2k$ all relevant values of $\omega$ coincide and consequently there is nothing to optimise.
Hence the error is given by $ z^{\omega_1(k,l)+\eps}$ with
\begin{align*}
  \omega_1(k,l) = \begin{cases}  \displaystyle{\frac{7(k+l)}{2k+6l+7kl}} & \text{ if } l \leq (3/4)k+1, \\[3mm] \displaystyle{\frac{4(k+l)+3}{2(k+1)(2l+1)}} & \text{ if } (3/4)k + 1 \leq  l \leq 2k, \\[3mm] \displaystyle{\frac{k+l}{kl+l}} & \text{ if } l \geq 2k. \end{cases}
\end{align*}

Recall now that during the analysis we possibly swapped the parameters $k$ and~$l$. Reversing that swap, we may return to the original assumption that $k\leq l$ and conclude that the optimal exponent is given by
    \begin{align*}
      \omega_2(k,l)  &=\begin{cases}
        \displaystyle{\frac{4(k+l)+3}{2(2k+1)(l+1)}}& \text{ if }  k \leq l \leq (4/3)(k-1),  \\[3mm]
         \displaystyle{\frac{7(k+l)}{6k+2l+7kl}} & \text{ if } (4/3)(k-1) \leq l.
      \end{cases}
      \end{align*}
Analogously, the first term of \eqref{eq:error-ab} is optimised by $z^{1/2+\eps}y^{-l(k-2)/(2k+2)}$ which coincides with $y^{1+\eps}$ for $y=z^{(k+l)/(2k+kl)}$, and the exponent exceeds $\omega_2(k,l)$ whenever~${l > 4k}$.
The proof is now complete on noting that the remaining terms $zy^{1-k+\eps}$ and $z^{1+\eps}y^{-l/2}$ in \eqref{eq:error1} coincide with $y$ when $y=z^{1/k}$ and $y=z^{2/(l+2)}$, respectively, both of which are dominated by the errors previously found.


\end{document}